\documentclass[a4paper,draft]{amsproc}
\usepackage{amssymb}
\usepackage{amscd} %% Package for commutative diagrams
\usepackage{snapshot}
\usepackage[dvips]{graphicx}
\usepackage{color}
\usepackage{subfigure}
\usepackage{rotating}
\usepackage{amsmath}
\usepackage{amsthm}
\usepackage{amsfonts}
\usepackage{stmaryrd}
\theoremstyle{plain}
 \newtheorem{theorem}{Theorem}[section]
 
 \newtheorem{lemma}{Lemma}[section]

\theoremstyle{remark}
 
 \numberwithin{equation}{section}

\renewcommand{\leq}{\leqslant}
\renewcommand{\geq}{\geqslant}

\title[]{And$\hat{o}$-Douglas type characterization of generalized conditional expectations, optional projections and predictable projections}


\author[Hong]{\bfseries Liang Hong}

\address{
Department of Mathematics \\ % \hfill (Received 00 00 2010)\\
Robert Morris University   \\ %\hfill (Revised  00 00 2010)\\
Moon, PA 15108, USA}
\email{hong@rmu.edu}

\begin{document}

\vspace{18mm}
\setcounter{page}{1}
\thispagestyle{empty}

\begin{abstract}
Generalized conditional expectations, optional projections and predictable projections of stochastic processes play important roles in the general theory of stochastic processes, semimartingale theory and stochastic calculus. They share some important properties with ordinary conditional expectations. While the characterization of ordinary conditional expectations has been studied by several authors, no similar work seems to have been done for these three concepts. This paper aims at undertaking this task by giving And$\hat{o}$-Douglas type characterization theorem for each of them.
\end{abstract}

\maketitle

\section{Introduction, notation, and setup}  %% Please avoid complicated formulas in titles
Properties of (ordinary) conditional expectation operators have been extensively studied in the literature. In particular, various authors have characterized conditional expectations. Earliest works along this line of research include \cite{Bahadur}, \cite{Moy}, \cite{Rota}, and \cite{Sidak}. \cite{Douglas} first characterized conditional expectations as contractive projections on $L_1$ spaces. \cite{AAB} provided a simple proof of the main theorem in \cite{Douglas}. \cite{Ando} extended the results of \cite{Douglas} to $L_p$ spaces. \cite{Pfanzagl} gave two characterizations of conditional expectations using expectation invariance. \cite{Jonge} applied the theory of Riesz spaces to characterize conditional expectations as order-continuous projections.  \cite{DHP} derived a general characterization theorem of conditional expectation operators.  The same result was derived in \cite{Mekler1} and \cite{Mekler2} independently. A refined account is given in \cite{AA}. \cite{Brunk1}, \cite{Brunk2} and \cite{Dykstra} studied and characterized conditional expectations with respect to a $\sigma$-lattice. Recently, \cite{Watson} generalized the And$\hat{o}$-Douglas theorem to the Riesz spaces.

To our best knowledge, no similar efforts have been made for the generalized conditional expectation (cf. Section I.4 of \cite{HWY}). This paper aims at filling this gap by proving an And$\hat{o}$-Douglas type characterization theorem for it. It seems that the concept of the generalized conditional expectation first appeared in \cite{Meyer}. Note that the generalized conditional expectation considered in \cite{Hornor} is different from the one we shall consider. Also, the generalized conditional expectation discussed in \cite{Yan1} is slightly more general than the one considered in this paper.

Optional projections and predictable projections of stochastic processes are two important concepts in the general theory of stochastic process; they are closely related to ordinary conditional expectations and generalized conditional expectations. On the theoretical side, their important connections to the semimartingale theory and stochastic calculus cannot be overemphasized (cf. \cite{DM}, \cite{HWY} and \cite{Protter}). On the applied side, \cite{Armerin} and \cite{NS} have successfully applied some related theoretical results to actuarial science and mathematical finance. It is known that optional and predictable projections share several properties with conditional expectations (cf. \cite{DM}, \cite{Dinculeanu}, \cite{HWY} and \cite{Protter}). Therefore, it is natural to ask whether it possible to characterize optional and predictable projections in a similar manner as in \cite{Ando} and \cite{Douglas} . We give an affirmative answer by giving characterization theorems for both.

Below we provide readers with the basic concepts which are necessary for this paper. For more details concerning general theory of stochastic processes, readers are referred to \cite{DM}, \cite{HWY} and \cite{Protter}; for further results concerning operator theory on Riesz spaces, readers are referred to \cite{AA}, \cite{AB}, and \cite{AB2}.

Here and throughout, all probabilistic objects are referred with respect to a fixed filtered probability space $(\Omega, \mathcal{F}, (\mathcal{F}_t)_{t\geq 0},  P)$ unless otherwise stated. The filtration $(\mathcal{F}_t)_{t\geq 0}$ is assumed to satisfy the usual conditions. We put
\begin{eqnarray*}
\mathcal{F}_{t+} &=& \cap_{s>t}\mathcal{F}_s, \quad t\geq 0, \nonumber\\
\mathcal{F}_{t-}&=&\cup_{s<t}\mathcal{F}_s=\sigma\left(\cup_{s<t}\mathcal{F}_s\right), \quad t>0,\\
\mathcal{F}_{0-} &=& \mathcal{F}_0, \nonumber\\
\mathcal{F}_{\infty} &=& \vee_{t\geq 0}\mathcal{F}_t=\sigma\left(\cup_{t\geq
                                                     0}\mathcal{F}_t\right), \nonumber \\
\mathcal{F}_{\infty-} &=& \mathcal{F}_{\infty}.
\end{eqnarray*}
Without loss of generality, two $P$-a.s. equal random variables will be considered equivalent. Let $\mathcal{G}$ be a sub-$\sigma$-field of $\mathcal{F}$. Recall that a random variable $\xi$ is said to be \emph{$\sigma$-integrable} with respect to $\mathcal{G}$ if there exists a sequence $\{\Omega_n\}\subset \mathcal{G}$ such that $\Omega_n\uparrow \Omega$ and $\xi1_{\Omega_n}$ is integrable for each $n$. If we put
\begin{equation}\label{1.1}
\mathcal{C}=\{A\in \mathcal{G}\mid \text{$\xi1_A$ is integrable}\},
\end{equation}
then there exists a $P$-a.s. unique real-valued $\mathcal{G}$-measurable random variable $\eta$ such that
\begin{equation}\label{1.2}
E[\xi1_A]=E[\eta1_A], \quad \forall A\in\mathcal{C}.
\end{equation}
$\eta$ is called the \emph{generalized conditional expectation} of $\xi$ with respect to $\mathcal{G}$ and is denoted as $E[\xi|\mathcal{G}]$ (cf. Theorem 1.17 of \cite{HWY}).
For a given sub-$\sigma$-field $\mathcal{G}$ of $\mathcal{F}$, $L_{\sigma}(\mathcal{G})$ will denote the family of all $\sigma$-integrable random variables with respect to $\mathcal{G}$; $L_1(\mathcal{G})$ will denote the family of all $\mathcal{G}$-measurable integrable random variables. Then the linear operator $E[\cdot\mid \mathcal{G}]: L_{\sigma}(\mathcal{G})\rightarrow L_{\sigma}(\mathcal{G})$ is called the \emph{generalized expectation operator} induced by $\mathcal{G}$.

Let $R_+=[0, \infty)$ and $\overline{R}=R\cup\{+\infty\}$. The Borel-$\sigma$-field on a set $E$ will be denoted by $\mathcal{B}(E)$. A \emph{stopping time} $T$ is an $\overline{R}$-valued random variable on $(\Omega, \mathcal{F})$ such that $[T\leq t]\in\mathcal{F}_t$ for all $t\geq 0$. The \emph{$\sigma$-field of events prior to $T$}, denoted by $\mathcal{F}_T$, is the set defined by
\begin{equation*}
\mathcal{F}_T=\{A\in\mathcal{F}_{\infty}\mid A\cap [T\leq t]\in\mathcal{F}_t\ \text{for all $t\geq 0$}\}.
\end{equation*}
Likewise, the \emph{$\sigma$-field of events strictly prior to $T$}, denoted by $\mathcal{F}_{T-}$, is the set defined by \begin{equation*}
\mathcal{F}_{T-}=\sigma(\{A\cap [t\leq T]\mid A\in \mathcal{F}_{t-} \text{ for all $t\geq 0$} \}).
\end{equation*}

We will use $\mathcal{T}$ to denote the collection of all stopping times, that is,
\begin{equation*}
\mathcal{T}=\{T\mid \text{ $T$ is a stopping time}\}.
\end{equation*}

A \emph{stochastic process} (or simply \emph{process}) $(X_t)_{t\geq 0}$ is a family of real-valued random variables indexed by $R_+$. It will often be denoted as $X$ for simplicity. We assume all processes are c$\grave{a}$dl$\grave{a}$g, that is, their sample paths are right-continuous and have left-hand limits.
A subset set $B$ of $\Omega\times R_+$ is said to be \emph{$P$-evanescent} if the projection of $B$ onto $\Omega$ is a $P$-null set. Two processes $X$ and $Y$ are said to be \emph{$P$-indistinguishable} if $\{(\omega, t)\mid X(\omega, t)\neq X(\omega, t)\}$ is a $P$-evanescent set. All $P$-indistinguishable processes will be treated as equivalent. A process $X$ is said to be \emph{measurable} if the mapping $(\omega, t)\mapsto X(\omega, t)$ is $\mathcal{F}\times\mathcal{B}(R_+)$-measurable. A process $X$ is said to be \emph{optional} (respectively \emph{predictable}) if it is $\mathcal{O}$-measurable (respectively $\mathcal{P}$-measurable), where $\mathcal{O}$ and $\mathcal{P}$ are the optional $\sigma$-field and predictable $\sigma$-field, respectively. A process $X$ is said to be \emph{progressively measurable} or \emph{progressive} if for every $t\geq 0$ the mapping $ (\omega, t)\mapsto X(\omega, t)$ restricted on $\Omega\times [0, t]$ is $\mathcal{F}_t\times \mathcal{B}([0, t])$-measurable.

A partially ordered set $X$ is called a \emph{lattice} if the infimum and supremum of any pair of elements in $X$ exist. A real vector space $X$ is called an \emph{ordered vector space} if its vector space structure is compatible with the order structure in a manner such that
\begin{enumerate}
  \item [(a)]if $x\leq y$, then $x+z\leq y+z$ for any $z\in X$;
  \item [(b)]if $x\leq y$, then $\alpha x\leq \alpha y$ for all $\alpha\geq 0$.
\end{enumerate}
An ordered vector space is called a \emph{Riesz space} (or a \emph{vector lattice}) if it is also a lattice at the same time. For any pair $x, y$ in a Riesz space, $x\vee y$ denotes and supremum of $\{x, y\}$, $x\wedge y$ denotes the infimum of $\{x, y\}$, and $|x|$ denotes $x\vee (-x)$. A vector subspace of a Riesz space is said to be a \emph{Riesz subspace} if it is closed under the lattice operation $\vee$. A subset $Y$ of a Riesz space $X$ is said to be \emph{solid} if $|x| \leq |y|$ and $y\in Y$ imply that $x\in Y$. A solid vector subspace of a Riesz space is called an \emph{ideal}. A Riesz space $X$ is said to be \emph{Dedekind complete} if every nonempty subset of $X$ that is bounded from above has a supremum. A Riesz space is said to be \emph{$\sigma$-Dedekind complete} if every nonempty countable subset that is bounded from above or bounded from below has a supremum or infimum, respectively. A Riesz space is said to have the \emph{countable super property} or \emph{order separable} if for every subset having a supremum contains an at most countable subset having the same supremum. A Dedekind complete Riesz space with the countable super property is said to be \emph{super Dedekind complete}. A net $(x_{\alpha})_{\alpha\in A}$ is said to be \emph{decreasing} if $\alpha\geq \beta$ implies $x_{\alpha}\leq x_{\beta}$. The notation $x_{\alpha}\downarrow x$ means $(x_{\alpha})_{\alpha\in A}$ is a decreasing net and the infimum of the set $\{x_{\alpha}\mid \alpha\in A\}$ is $x$. A net $(x_{\alpha})_{\alpha\in A}$ in a Riesz space $X$ is said to be \emph{order-convergent} to an element $x\in X$, often written as $x_{\alpha}\xrightarrow{o} x$,  if there exists another net $(y_{\alpha})_{\alpha\in A}$ in $X$ such that $|x_{\alpha}-x|\leq y_{\alpha}\downarrow 0$.

Let $T$ be a linear operator on a vector space $X$. A vector subspace $Y$ of $X$ is said to be \emph{$T$-invariant} if $T(Y)\subset Y$. In this case, we say $T$ leaves $Y$ invariant. A linear operator $T$ on a vector space $X$ is said to be a \emph{projection} if $T^2=T$. A linear operator $T$ between two Riesz spaces $X$ and $Y$ is said to be \emph{positive} if $x\in X$ and $x\geq 0$ implies $T(x)\geq 0$; $T$ is said to be \emph{strictly positive} if $x\in X$ and $x>0$ implies $T(x)>0$; $T$ is said to be \emph{order-continuous} if $x_{\alpha}\xrightarrow{o} 0$ in $X$ implies $T(x_{\alpha})\xrightarrow{o} 0$ in $Y$. A positive projection that leaves the constant function $1$ invariant is called a \emph{Markov projection}. A linear operator $T$ on a vector space $V$ is called an \emph{averaging operator} if $T(y T(x))=T(y)T(x)$ for any pair $x, y \in V$ such that $y T(x)\in V$.

The remainder of this paper is organized as follows: Section 2 gives a characterization theorem of generalized conditional expectations; Section 3 characterizes optional projections; Section 4 proves a characterization theorem of predictable projections.

\section{Characterization of generalized conditional expectations}
Theorem \ref{theorem2.1} gives some important properties of the generalized conditional expectation in terms of operator theory. First, we need a lemma.

\begin{lemma}\label{lemma2.1}
The space $L_{\sigma}(\mathcal{G})$ is super Dedekind complete.
\end{lemma}

\begin{proof}
Let $L_0(\mathcal{G})$ denote the space of all $\mathcal{G}$-measurable random variables.
Since a probability measure is $\sigma$-finite, $L_{0}(\mathcal{G})$ is a super Dedekind complete Riesz space. It is clear that $L_{\sigma}(\mathcal{G})$ is a vector subspace of $L_0(\mathcal{G})$. Suppose $\xi, \eta\in L_0(\mathcal{G})$ such that $|\eta|\leq |\xi|$ and $\xi\in L_{\sigma}(\mathcal{G})$. Then there exists a positive element $\zeta\in L_0(\mathcal{G})$ such that $\xi\zeta$ is integrable. It follows that $\eta \zeta$ is integrable too. Hence, $\eta\in L_{\sigma}(\mathcal{G})$. This shows that $L_{\sigma}(\mathcal{G})$ is an ideal of $L_0(\mathcal{G})$. Therefore, $L_{\sigma}(\mathcal{G})$ is super Dedekind complete.\\
\end{proof}

\begin{theorem}\label{theorem2.1}
Let $(\Omega, \mathcal{F}, P)$ be a probability space and $\mathcal{G}$ be a sub-$\sigma$-field of $\mathcal{F}$. Then the generalized conditional expectation operator $E[\cdot\mid \mathcal{G} ]: L_{\sigma}(\mathcal{G})\rightarrow L_{\sigma}(\mathcal{G})$ is a strictly positive and order-continuous Markov projection. Moreover, $E[\cdot\mid \mathcal{G} ]$ is an averaging operator and leaves the space $L_{\sigma}^p=\{\xi\mid \text{$\xi^p$ is $\sigma$-integrable}\}$ invariant for $1\leq p\leq \infty$.
\end{theorem}

\begin{proof}
Equation (\ref{1.2}) shows that $E[\cdot\mid\mathcal{G}]$ is strictly positive. Since the dominated convergence theorem holds for the generalized condition expectation, we know that if $\xi_n\downarrow 0$ in $L_{\sigma}(\mathcal{G})$, then $E[\xi_n\mid \mathcal{G}]\downarrow 0$ in $L_{\sigma}(\mathcal{G})$. According to Lemma \ref{lemma2.1}, the Riesz space $L_{\sigma}(\mathcal{G})$ is super Dedekind complete. It follows that the operator $E[\cdot \mid \mathcal{G}]$ is order-continuous. For any $\xi\in L_{\sigma}(\mathcal{G})$, the generalized conditional expectation $E(\xi\mid \mathcal{G})$ is $\mathcal{G}$-measurable. Therefore, the smoothing property of the generalized conditional expectation implies $E[E[\xi\mid \mathcal{G}]\mid \mathcal{G}]=E[\xi\mid \mathcal{G}]$. This shows that $E[\cdot\mid\mathcal{G}]^2=E[\cdot \mid \mathcal{G}]$, i.e., $E[\cdot\mid \mathcal{G}]$ is a projection. It is clear that $E[1_{\Omega}\mid \mathcal{G}]=1_{\Omega}$. Hence, $E[\cdot\mid \mathcal{G}]$ is a Markov projection.

To see the second statement, take any pair $\xi, \eta\in L_{\sigma}(\mathcal{G})$ with $\eta E[\xi\mid\mathcal{G}] \in L_{\sigma}(\mathcal{G})$. Since $E[\xi\mid\mathcal{G}]$ is $\mathcal{G}$-measurable, the smoothing property of the generalized conditional expectation implies $E[\eta E[\xi\mid\mathcal{G}]\mid \mathcal{G}]=E[\xi\mid\mathcal{G}]E[\eta\mid\mathcal{G}]$. It remains to show $E[\cdot\mid\mathcal{G}]$ leaves $L_{\sigma}^p$ invariant. The case where $p=1$ or $p=\infty$ is trivial. Assume $1<p<\infty$. Let $\xi\in L_{\sigma}^p$. Then Jensen's inequality for generalized conditional expectations implies $(E[|\xi|\mid \mathcal{G}])^p\leq  E[|\xi|^p\mid \mathcal{G}]$. Since $\xi^p$ is $\sigma$-integrable, so is $E[|\xi|^p\mid \mathcal{G}]$. It follows that $E[|\xi|\mid \mathcal{G}]\in L_{\sigma}^p(\mathcal{G})$.
\end{proof}
\noindent \textbf{Remark.} The ordinary conditional expectation operator is evidently a contractive projection on $L_1$ spaces. However, this result cannot be extended to the generalized conditional expectation operator because the generalized conditional expectation is only $\sigma$-integrable and the $L_1$ norm in general does not apply to the space $L_{\sigma}(\mathcal{G})$.\\

The following Douglas' theorem (Corollary 11 in \cite{Douglas}) is the classical characterization theorem for ordinary conditional expectation operators.

\begin{theorem}\label{theorem2.2}
Let $(\Omega, \mathcal{F}, P)$ be a probability space and $T: L_1(\mathcal{F})\rightarrow L_1(\mathcal{F})$ be a linear operator. Then the following two statements are equivalent.
\begin{enumerate}
  \item [(i)]$T$ is a contractive Markov projection.
  \item [(ii)]$T$ is an ordinary conditional expectation operator, that is, there exists a sub-$\sigma$-field $\mathcal{G}$ of $\mathcal{F}$ such that $T$ is the ordinary conditional expectation operator $E[\cdot\mid\mathcal{G}]$.
\end{enumerate}
\end{theorem}

Next, we prove a characterization theorem for the generalized conditional expectation operator. This theorem generalizes Dougals' theorem to the case of generalized conditional expectation operators.

\begin{theorem}\label{theorem2.3}
Let $(\Omega, \mathcal{F}, P)$ be a probability space and $\mathcal{G}$ be a sub-$\sigma$-field of $\mathcal{F}$. Then the following two statements are equivalent.
\begin{enumerate}
  \item [(i)]$T$ is a Markov projection on $L_{\sigma}(\mathcal{G})$ such that
             \begin{enumerate}
               \item [(a)]$T(\xi)$ is $\mathcal{G}$-measurable and $T(\xi1_A)=T(\xi)1_A$ for $\xi\in L_{\sigma}(\mathcal{G})$ and $A\in\mathcal{G}$;
               \item [(b)]the restriction of $T$ on $L_1(\mathcal{\mathcal{G}})$ is a contractive projection.
             \end{enumerate}
  \item [(ii)]$T$ is the generalized conditional expectation operator induced by $\mathcal{G}$.
\end{enumerate}
\end{theorem}
\begin{proof}
For a given sub-$\sigma$-field $\mathcal{G}$ of $\mathcal{F}$, we will use $E_g[\cdot \mid \mathcal{G}]$ to denote the generalized conditional expectation operator and $E_o[\cdot \mid \mathcal{G}]$ to denote the ordinary conditional expectation operator.

(i) $\Longrightarrow$ (ii). Suppose statement (i) holds. Then property (b) and Theorem \ref{theorem2.2} imply that there exists a sub-$\sigma$-field $\mathcal{H}\subset\mathcal{G}$ such that $T=E_o[\cdot \mid \mathcal{H}]$ on $L_1(\mathcal{H})$. Take any $\xi\in L_{\sigma}(\mathcal{H})$ and a set $A\in \mathcal{C}$, where the class $\mathcal{C}$ is defined in Equation (\ref{1.1}). Then $\xi\in L_{\sigma}(\mathcal{G})$ and $A\in\mathcal{G}$; hence we have $T(\xi)1_A=T(\xi1_A)=E_o[\xi1_A \mid \mathcal{H}]$. It follows that $E[T(\xi)1_A]=E[\xi1_A]$.
Since $T(\xi)$ is $\mathcal{H}$-measurable by property (a), the uniqueness of generalized conditional expectations implies $T(\xi)=E_g[\xi\mid \mathcal{H}]$, $P$-a.s., i.e., $T=E_g[\cdot \mid \mathcal{H}]$ on $L_{\sigma}(\mathcal{H})$. Next, take any $B\in \mathcal{G}$. Since $T=E_o[\cdot \mid \mathcal{H}]$ on $L_1(\mathcal{H})$, $T(1_B)$ is $\mathcal{H}$-measurable. Applying property (a) with $\xi=1$, we see that $1_B$ is $\mathcal{H}$-measurable, implying $\mathcal{H}=\mathcal{G}$. Therefore, we have $T=E_g[\cdot\mid \mathcal{G}]$.

(ii) $\Longrightarrow$ (i). Assume statement (ii) holds. We need to show that $T$ satisfies properties (a) and (b). Property (a) follows from the smoothing property of generalized conditional expectations. Property (b) follows from Theorem \ref{theorem2.2}.
\end{proof}
\noindent \textbf{Remark.} For a given sub-$\sigma$-field $\mathcal{G}$, the domain of $E_o[\cdot\mid \mathcal{G}]$ is always $L_1(\mathcal{F})$ regardless of the choice of $\mathcal{G}$. However, the domain of $E_g[\cdot\mid \mathcal{G}]$ is $L_{\sigma}(\mathcal{G})$ which depends on the choice of $\mathcal{G}$. This explains why the form of Theorem \ref{theorem2.3} is slightly different from that of Theorem \ref{theorem2.2}.

\section{Characterization of optional projections}
First, we recall a well-known result from the general theory of stochastic process (cf. \cite{HWY}). It is essentially an existence theorem of optional projections.
\begin{theorem}\label{theorem3.0}
Let $X$ be a measurable process such that the random variable $X_T1_{[T<\infty]}$ is $\sigma$-integrable w.r.t. $\mathcal{F}_T$ for every $T\in\mathcal{T}$. Then there exists a unique optional process $^oX$ such that for every $T\in\mathcal{T}$,
\begin{equation}\label{3.0}
E[X_T1_{[T<\infty]}\mid \mathcal{F}_T]={^oX}_T1_{[T<\infty]}.
\end{equation}
In this case, we say the optional projection $^oX$ of $X$ exists and is finite.
\end{theorem}

To characterize the optional projection, we define
\begin{eqnarray*}
L_m &=& \{X\mid \text{ $X$  is a measurable process}\},\\
L_{o}&=& \{X\mid \text{ $X$  is an optional process}\},\\
L_{op} &=& \{X\mid \text{ $X\in L_m$ and $X_T1_{[T<\infty]}$ is $\sigma$-integrable w.r.t $\mathcal{F}_T$ for every $T\in\mathcal{T}$}\}.
\end{eqnarray*}
Since the optional $\sigma$-field is contained in the progressive $\sigma$-field and a progressive process is measurable, we have $L_{o}\subset L_{m}$.

\begin{lemma}\label{lemma3.1}
$L_{op}$  is a super Dedekind complete Riesz space.
\end{lemma}

\begin{proof}
Let $\mu_L$ denote the Lebesgue measure on $R$. The product measure $P\times \mu_L$ on $(\Omega\times R_+, \mathcal{F}\times\mathcal{B}(R_+))$ is evidently $\sigma$-finite. Thus, $L_{m}$ is a super Dedekind complete Riesz space.

For any two processes $X$ and $Y$ in $L_{op}$, $X$ and $Y$ are both measurable mappings from $(\Omega\times R_+, \mathcal{F}\times\mathcal{B}(R_+))$ to $(R, \mathcal{B}(R))$. Therefore, $X\wedge Y$ is also a measurable mapping from $(\Omega\times R_+, \mathcal{F}\times\mathcal{B}(R_+))$ to $(R, \mathcal{B}(R))$, that is, $X\wedge Y$ is a measurable process. Moreover, $X_T1_{[T<\infty]}$ and $Y_T1_{[T<\infty]}$ are both $\sigma$-integrable with respect to $\mathcal{F}_T$ for all $T\in\mathcal{T}$. It follows from
\begin{equation*}
(X\wedge Y)_T1_{[T<\infty]}\leq X_T 1_{[T<\infty]}+Y_T1_{[T<\infty]}
\end{equation*}
that $(X\wedge Y)_T1_{[T<\infty]}$ is also $\sigma$-integrable with respect to $\mathcal{F}_T$. Thus, $L_{op}$ is a Riesz subspace of $L_m$.

Next, consider two processes $X$ and $Y$ in $L_m$ such that $|X|\leq |Y|$ and $Y\in L_{op}$. Let $T$ be a stopping time. Then $Y_T1_{[T<\infty]}$ is $\sigma$-integrable with respect to $\mathcal{F}_T$; hence $X$ is $\sigma$-integrable with respect to $\mathcal{F}_T$ too. This implies that $L_{op}$ is an ideal of $L_m$. Hence, $L_{op}$ is super Dedekind complete.
\end{proof}

For every process $X$ in $L_{op}$, its optional projection $^oX$ exists and is finite. Thus, $T: X\mapsto {^oX}$ is a linear operator from $L_{op}$ to $L_{m}$. Indeed, $T$ is a projection on $L_{op}$ as we shall see next.

\begin{theorem}\label{theorem3.1}
The linear operator $T: X\mapsto  {^oX}$ is an order-continuous strictly positive Markov projection on $L_{op}$. Moreover, $T$ is also an averaging operator.
\end{theorem}

\begin{proof}
Take a process $X$ in $L_{op}$ and a stopping time $T$ in $\mathcal{T}$. Then $X_T1_{[T<\infty]}$ is $\sigma$-integrable with respect to $\mathcal{F}_T$. Thus, there exists an increasing sequence $\{\Omega_n\}$ of sets in $\mathcal{F}_T$ such that $\Omega_n \uparrow \Omega$ and $X_T1_{[T<\infty]}1_{\Omega_n}$ is integrable for each $n$. It follows from (\ref{3.0}) and Jensen's inequality that
\begin{eqnarray*}
& & E|{^oX}1_{[T<\infty]}1_{\Omega_n}| \\
&=& E\left[|E[X_T1_{[T<\infty]}\mid \mathcal{F}_T1_{\Omega_n}]|\right]\\
&\leq& E\left[E|[X_T 1_{[T<\infty]}1_{\Omega_n}|\mid \mathcal{F}_T]|\right]\\
&=& E[|X_T1_{[T<\infty]}1_{\Omega_n}|]<\infty.
\end{eqnarray*}
This shows that $^oX$ is $\sigma$-integrable with respect to $\mathcal{F}_T$; hence $L_{op}$ is $T$-invariant, i.e., $T(X)\in L_{op}$.

Moreover, for every $X$ in $L_{op}$ we have
\begin{eqnarray*}
& & E[{^oX}1_{[T<\infty]}\mid \mathcal{F}_T] \\
&=& E[X_T1_{[T<\infty]}\mid \mathcal{F}_T]\\
&=& {^oX}_T1_{[T<\infty]}.
\end{eqnarray*}
This shows $^o(^oX)={^oX}$, that is $T(T(X))=T(X)$. Hence $T$ is a projection on $L_{op}$.

It is clear from (\ref{3.0}) that $T(1)=1$, i.e., $^o1=1$ and $T$ is a positive operator. To see $T$ is strictly positive, we assume $X>0$ and $^oX=0$. Then for every stopping time $T$ (in particular, for every $T=t>0$), (\ref{3.0}) shows that
\begin{equation*}
E[X_T1_{[T<\infty]}]=0,
\end{equation*}
contradicting the hypothesis $X>0$.

Next, we show $T$ is order-continuous. By Lemma \ref{lemma3.1}, it suffices to show that $X^n\downarrow 0$ implies $T(X^n)\downarrow 0$, that is, $X^n\downarrow 0$ implies ${^oX^n}\downarrow 0$. To this end, let $(X^n)$ be a sequence in $L_{op}$ such that $X^n\downarrow 0$. For any $T\in \mathcal{T}$, we apply the monotone convergence theorem for generalized conditional expectation to the sequence $X^1-X^n$ to conclude
\begin{equation*}
\lim_{n\rightarrow \infty} {^oX}^n_T1_{[T<\infty]}=\lim_{n\rightarrow \infty}E[X^n_T1_{[T<\infty]}\mid\mathcal{F}_T]=E\left[\lim_{n\rightarrow\infty} X_T^n1_{[T<\infty]}\mid \mathcal{F}_T\right]=0.
\end{equation*}

Finally, we show $T$ is an averaging operator. Let $X$ and $Y$ be two processes in $L_{op}$ such that $YT(X)=Y{^oX}\in L_{op}$. Then the smoothness of optional projections implies
\begin{equation*}
{^o(Y{^oX})}=(^oY)(^o X),
\end{equation*}
that is, $T(YT(X))=T(Y)T(X)$.
\end{proof}
\noindent \textbf{Remark.} Alternatively, we can see that $T$ is a projection on $L_{op}$ as follows. Since $T(X)$ is optional, it is progressively measurable; hence $(T(X))_S1_{[S<\infty]}$ is $\mathcal{F}_S$-measurable for any $S\in \mathcal{T}$. It follows that $(T(X))_S1_{[S<\infty]}$ is $\sigma$-integrable w.r.t. $\mathcal{F}_S$, implying that $T(X)\in L_{op}$. This also shows that $L_o\subset L_{op}$. \\

The next theorem characterizes the optional projection.

\begin{theorem}[Characterization of optional projections]\label{theorem3.2}
For any linear operator $T: L_{op}\rightarrow L_{op}$, the following two statements are equivalent.
\begin{enumerate}
  \item [(1)]$T$ is the optional projection operator $T: X\mapsto {^oX}$.
  \item [(2)]$T$ is an order-continuous Markov projection on $L_{op}$.
\end{enumerate}
\end{theorem}

\begin{proof}
(1)$\Longrightarrow$ (2) follows from Theorem \ref{theorem3.1}. It remains to show (2)$\Longrightarrow$ (1).

Since $T$ is Markov projection on $L_{op}$, $T(1)=1$, where $1$ is the constant process $(1)_{t\geq 0}$. This implies $T(1)=(E[1\mid \mathcal{F}_t])_{t\geq 0}$. Let $\xi$ be any integrable random variable on $(\Omega, \mathcal{F}, P)$. Choose a sequence  $(\xi^n)$ of simple random variables such that $|\xi^n| \leq |\xi|$ and $\xi^n \uparrow\xi$. By the linearity and order-continuity of $T$ as well as the dominated convergence theorem for conditional expectations, we have
\begin{eqnarray*}
T((\xi)_t) &=& \lim_{n\rightarrow \infty} T((\xi^n)_t)\\
           &=&\lim_{n\rightarrow \infty} (E[\xi^n\mid \mathcal{F}_t])_{t\geq 0}\\
           &=&(E[\xi\mid\mathcal{F}_t])_{t\geq 0}.
\end{eqnarray*}

Next, suppose $X\in L_{op}$ and $X=\xi1_{\llbracket s_1, s_2\llbracket}$ where $\xi$ is an integrable random variable, $0\leq s_1<s_2\leq +\infty$ and $\llbracket s_1, s_2\llbracket$ is the stochastic interval from $s_1$ to $s_2$. Since $T$ is an averaging operator, we have
\begin{eqnarray*}
T(X) &=& T(\xi 1_{\llbracket s_1, s_2\llbracket})\\
     &=& 1_{\llbracket s_1, s_2\llbracket} T((\xi)_{t\geq 0}) \\
     &=& 1_{\llbracket s_1, s_2\llbracket}(E[\xi\mid \mathcal{F}_t])_{t\geq 0}.
\end{eqnarray*}
This implies that $T(X)={^oX}$ for every simple process $X$ in $L_{op}$.

If $X\in L_{op}$ and $X\geq 0$, then we may choose a sequence of nonnegative simple process $X^{n}\in L_{op}$ such that $X^n\uparrow X$. Then the order-continuity of $T$ implies
\begin{equation*}
T(X)=\lim_{n\rightarrow \infty} T(X^n)=\lim_{n\rightarrow \infty} {^oX^n}={^oX}. \\
\end{equation*}

For a general $X\in L_{op}$, $X^+$ and $X^-$ are both nonnegative elements of $L_{op}$. Therefore, we have
\begin{equation*}
T(X)=T(X^+-X^-)=T(X^+)-T(X^-)={^o(X^+)}-{^o(X^-)}={^oX}.
\end{equation*}
This establishes that $T(X)={^oX}$ for every $X\in L_{op}$.
\end{proof}

\section{Characterization of predictable projections}
Following the same line of reasoning in the previous section, we can prove a characterization theorem of predictable projections. Hence, we will state the results without proofs.

We put
\begin{eqnarray*}
\mathcal{T}_p &=& \{T\mid \text{ $T$ is a predictable time}\}, \\
L_{p}&=& \{X\mid \text{ $X$  is a predictable process}\},\\
L_{pp} &=& \{X\mid \text{ $X\in L_m$ and $X_T1_{[T<\infty]}$ is $\sigma$-integrable w.r.t $\mathcal{F}_{T-}$ for every $T\in\mathcal{T}_p$}\}.\\
\end{eqnarray*}

Recall the following result from the general theory of stochastic processes (cf. \cite{HWY}).
\begin{theorem}\label{theorem4.0}
Let $X$ be a measurable process such that the random variable $X_T1_{[T<\infty]}$ is $\sigma$-integdrable w.r.t. $\mathcal{F}_{T-}$ for every $T\in\mathcal{T}_p$. Then there exists a unique predictable process $^pX$ such that for every $T\in\mathcal{T}_p$,
\begin{equation}\label{4.0}
E[X_T1_{[T<\infty]}\mid \mathcal{F}_{T-}]={^pX}_T1_{[T<\infty]}.
\end{equation}
In this case, we say the predictable projection $^pX$ of $X$ exists and is finite.
\end{theorem}

Similar to Lemma \ref{lemma4.1}, we have the following lemma.

\begin{lemma}\label{lemma4.1}
$L_{pp}$  is a super Dedekind complete Riesz space.
\end{lemma}

For every process $X$ in $L_{pp}$, its predictable projection $^pX$ exists and is finite. Thus, $\widetilde{T}: X\mapsto {^pX}$ is a linear operator from $L_{pp}$ to $L_{m}$. Indeed, $\widetilde{T}$ is a projection on $L_{pp}$ as we shall show next.

\begin{theorem}\label{theorem4.1}
The linear operator $\widetilde{T}: X\mapsto  {^pX}$ is an order-continuous strictly positive Markov projection on $L_{pp}$. Moreover, $\widetilde{T}$ is also an averaging operator.
\end{theorem}

The next result shows that the properties of $\widetilde{T}$ stated in Theorem \ref{theorem4.1} characterizes the predictable projection.

\begin{theorem}[Characterization of predictable projections]
For any linear operator $\widetilde{T}: L_{pp}\rightarrow L_{pp}$, the following two statements are equivalent.
\begin{enumerate}
  \item [(1)]$\widetilde{T}$ is the predictable projection operator $\widetilde{T}: X\mapsto {^pX}$.
  \item [(2)]$\widetilde{T}$ is an order-continuous Markov projection on $L_{pp}$.
\end{enumerate}
\end{theorem}

%\newpage

\bibliographystyle{amsplain}

\end{document}